\documentclass[12pt]{article}   	
\usepackage{geometry}                		
\usepackage{graphicx}				
\usepackage{float}					
\usepackage{cancel}
\usepackage{centernot}
\usepackage{bbm}
\usepackage{listings}
\usepackage{fancyvrb}
\usepackage{hyperref}
\usepackage{textgreek}
\usepackage[usenames,dvipsnames]{xcolor}
\usepackage{tikz}
\usetikzlibrary{
  graphs,
  graphs.standard
}
\usepackage{asymptote}
\usepackage{amssymb,amsmath,amsfonts,amsthm}
\usepackage{setspace}
\usepackage{booktabs}
\usepackage{titling}
\usepackage{hyperref}

\newtheorem{theorem}{Theorem}[section]
\newtheorem{lemma}[theorem]{Lemma}
\newtheorem{corollary}{Corollary}[theorem]
\newtheorem{definition}{Definition}[section]

\title{Maximum Size of a Family of Pairwise Graph-Different Permutations}
\author{Louis Golowich\thanks{MIT PRIMES, Department of Mathematics, Massachusetts Institute of Technology.}, 
Chiheon Kim\thanks{Department of Mathematics,
Massachusetts Institute of Technology,
Cambridge, MA 02139-4307.
Email: {\tt chiheonk@math.mit.edu}.}, 
Richard Zhou$^*$}
\date{\today}							

\begin{document}
\maketitle
\begin{abstract}
\normalsize
Two permutations of the vertices of a graph $G$ are called $G$-different if there exists an index $i$ such that $i$-th entry of the two permutations form an edge in $G$. We bound or determine the maximum size of a family of pairwise $G$-different permutations for various graphs $G$. 
We show that for all balanced bipartite graphs $G$ of order $n$ with minimum degree $n/2 - o(n)$, the maximum number of pairwise $G$-different permutations of the vertices of $G$ is $2^{(1-o(1))n}$. We also present examples of bipartite graphs $G$ with maximum degree $O(\log n)$ that have this property. We explore the problem of bounding the maximum size of a family of pairwise graph-different permutations when an unlimited number of disjoint vertices is added to a given graph. We determine this exact value for the graph of 2 disjoint edges, and present some asymptotic bounds relating to this value for graphs consisting of the union of $n/2$ disjoint edges.
\end{abstract}

\section{Introduction}
For any graph $G$, let two permutations of the vertices of $G$ be $G$-{\it different} if there exists some index $i$ such that the $i$-th entry of the two permutations form an edge in $G$. Let $F(G)$ be the maximum size of a family of pairwise $G$-different permutations of the vertices of $G$.
The value of $F(G)$ has been studied for many graphs $G$. One of the most studied such graphs is the path on $n$ vertices $P_n$ (pairs of $P_n$-different permutations are also called {\it colliding} permutations in \cite{korner_pairwise_2006}). K\"{o}rner and Malvenuto \cite{korner_pairwise_2006} conjectured that $F(P_n) = {n \choose \lfloor n/2 \rfloor}$. The authors' results implied that 
\begin{equation}
\label{eq:know_upper_bound}
F(G) \leq {n \choose \lfloor n/2 \rfloor}
\end{equation}
for all $n$-vertex balanced bipartite graphs $G$, and $F(K_{\lfloor n/2 \rfloor, \lceil n/2 \rceil}) = {n \choose \lfloor n/2 \rfloor}$, where $K_{\lfloor n/2 \rfloor, \lceil n/2 \rceil}$ is the complete balanced bipartite graph on $n$ vertices. The current asymptotic bounds on $F(P_n)$ stand at $$1.81 \leq \lim_{n \rightarrow \infty} (F(P_n))^{1/n} \leq 2;$$ the lower bound was shown in \cite{brightwell_permutation_2010}.

It is conjectured that $F(P_n) = F(K_{\lfloor n/2 \rfloor, \lceil n/2 \rceil})$, which is surprising; the complete balanced bipartite graph has many more edges than the path, which is one of the sparsest connected balanced bipartite graphs. Therefore we investigate $F(G)$ for balanced bipartite graphs $G$ more dense than the path but less dense than the complete balanced bipartite graph.

In this paper, we present new bounds on $F(G)$ for various bipartite graphs $G$, thereby potentially making progress towards determining this value for the path. We show that for all dense enough $n$-vertex balanced bipartite graphs $G$, $F(G)$ is near $F(K_{\lfloor n/2 \rfloor, \lceil n/2 \rceil})$. We also present a smaller family of much sparser bipartite graphs, which have average degree $O(\log n)$, for which this growth holds. In comparison, the path graph has average degree approximately 2. We investigate the properties of families of pairwise graph-different permutations where arbitrarily many disjoint vertices are added to a graph. We develop new methods for bounding this quantity for the matching graph and determine its exact value for the 4-vertex matching graph (the graph of 2 independent edges).

In related work, K\"{o}rner, Malvenuto, and Simonyi \cite{korner_graph-different_2008} bounded $F(G)$ for various graphs $G$ with arbitrarily many isolated vertices, and completely determined this value for stars. Cohen and Malvenuto \cite{cohen_cyclic_2013} presented bounds on $F(C_n)$, where $C_n$ denotes the $n$-vertex cycle. Their bounds are similar to the current bounds on $F(P_n)$. K\"{o}rner, Simonyi, and Sinaimeri \cite{korner_types_2009} investigated distance graphs, as well as specific graphs $G$ with $n$ vertices such that $F(G)$ does not grow exponentially with $n$, in contrast to the majority of the results in this field. Frankl and Deza \cite{frankl_maximum_1977} looked at a slightly different problem, in which they bounded the maximum number of pairwise $t$-intersecting permutations, where two permutations are $t$-intersecting if they share $t$ common positions.

The organization of this paper is as follows. In Section \ref{sec:dense_bipartite}, we present classes of balanced bipartite graphs for which $F$ is near the upper bound given in (\ref{eq:know_upper_bound}). In Section \ref{sec:sparse_bipartite}, we investigate the properties of families of pairwise matching-different permutations. We present implications and potential future extensions of our work in Section \ref{sec:conclusion}.

\section{Dense Bipartite Graphs: Lower Bounds}
\label{sec:dense_bipartite}

In Section \ref{subsec:max_deg_comp}, we present lower bounds on $F$ for $n$-vertex bipartite graphs with three different maximum degrees of the bipartite complement; namely, $1$, any positive constant, and $o(n)$. In Section \ref{subsec:union_disjoint}, we present lower bounds for the graph consisting of the union of small disjoint balanced bipartite graphs.

\subsection{Bipartite Graphs with Specified Maximum Degree of Complement}
\label{subsec:max_deg_comp}
It was shown in \cite{korner_pairwise_2006} that $F(P_n) \leq {n \choose \lfloor n/2 \rfloor}$; from their proof it immediately follows that $F(K_{a, n-a}) = {n \choose a}$. The following trivial lemma shows that this upper bound applies to all bipartite graphs.
\begin{lemma}
\label{lem:subgraph_is_less}
If $G$ is a subgraph of $H$, then $F(G) \leq F(H)$.
\end{lemma}
\begin{proof}
Any pair of $G$-different permutations must also be $H$-different by definition; therefore any family of pairwise $G$-different permutations is also pairwise $H$-different.
\end{proof}
\begin{corollary}
\label{cor:bipartite_upper_bound}
If $G$ is subgraph of $K_{a,n-a}$, then $F(G) \leq {n \choose a}$.
\end{corollary}

Because it is conjectured that the upper bound in Corollary \ref{cor:bipartite_upper_bound} is tight for the $n$-vertex path, it is natural to try to show that this bound is tight for other classes of non-complete bipartite graphs. We present such a class of graphs in the following result.

Let $G(n, a)$ be the graph obtained by removing a maximal matching from $K_{a,n-a}$. (By this definition, $G(n, n/2)$ is the crown graph on $n$ vertices.) We use induction on $n$ and $a$ to determine $F(G(n,a))$ for all $n$ and $a$ in the below theorem.
\begin{theorem}
\label{thm:matching_removed}
For all nonnegative integers $a \leq n$,
\[ F(G(n,a)) =
\begin{cases} 
1 & n < 3 \\
{n \choose a} & n \geq 3. \\
\end{cases}
\]
\end{theorem}

\begin{proof}
If $n < 3$, there are at most 2 vertices in $G(n, a)$, so the graph does not have any edges by definition. Therefore no two permutations are $G(n,a)$-different, so $F(G(n,a)) = 1$. We now assume that $n \geq 3$.
It suffices to show that $F(G(n,a)) \geq {n \choose a}$, as $F(G(n,a)) \leq{n \choose a}$ by Corollary \ref{cor:bipartite_upper_bound}. We prove the result by induction. For the base case, note that for any nonnegative integer $n$, $F(G(n,0)) = F(G(n,n)) = 1 = {n \choose 0} = {n \choose n}$. This is because $G(n,0)$ and $G(n,n)$ both have no edges, so no two permutations are $G(n,0)$-different or $G(n,n)$-different. Additionally, $F(G(3,1)) = F(G(3,2)) = 3 = {3 \choose 1} = {3 \choose 2}$. This is because $G(3,1)$ and $G(3,2)$ each have 3 vertices and 1 edge, so it suffices to show that $F(H) \geq 3$, where $H$ is a graph with vertices labeled 1,2,3 and with an edge between 1 and 2. The 3 permutations
\begin{equation}
\begin{matrix}
1 & 2 & 3 \\
3 & 1 & 2 \\
2 & 3 & 1 \\
\end{matrix}
\nonumber
\end{equation}
are pairwise $H$-different, so $F(H) \geq 3$, and therefore $F(H) = 3$.
For the inductive step, assume $n > 3$, $0 < a < n$, and $F(G(n-1, d)) = {n-1 \choose d}$ for all $0 \leq d \leq n-1$. It follows that $G(n,a)$ is not an empty graph, so let $x$ and $y$ be vertices in the first and second subsets of $G(n,a)$ respectively such that there is an edge between $x$ and $y$. If $x$ is removed from $G(n,a)$, the resulting graph is either $G(n-1,a-1)$ or a supergraph of $G(n-1,a-1)$. Likewise, if $y$ is removed from $G(n,a)$, the resulting graph is either $G(n-1,a)$ or a supergraph of $G(n-1,a)$. Then by the inductive hypothesis, there exists a family of at least $F(G(n-1,a-1))$ permutations of $V(G(n,a)) - \{x\}$ that are pairwise $G(n,a)$-different, and there exists a family of at least $F(G(n-1,a))$ permutations of $V(G(n,a)) - \{y\}$ that are pairwise $G(n,a)$-different; let these families be $\mathcal{F}_x$ and $\mathcal{F}_y$ respectively. Let $\mathcal{F}$ be the family that consists of the union of $x$ concatenated to all elements of $\mathcal{F}_x$ and $y$ concatenated to all elements of $\mathcal{F}_y$. Then $\mathcal{F}$ is pairwise $G(n,a)$-different, so $$F(G(n,a)) \geq |\mathcal{F}| = |\mathcal{F}_x| + |\mathcal{F}_y| = F(G(n-1,a-1)) + F(G(n-1,a)).$$ By this induction, $F(G(n,a)) = {n \choose a}$ for $n \geq 3$, as ${n \choose a} = {n-1 \choose a-1} + {n-1 \choose a}= F(G(n-1,a-1)) + F(G(n-1,a))$.
\end{proof}

Therefore for all $n \geq 3$, there exist non-complete bipartite graphs on $n$ vertices that are subgraphs of $K_{a,n-a}$ for which the upper bound of ${n \choose a}$ is exactly equal to $F$. However, the graphs $G(n,a)$ considered in the above theorem are such that the maximum degrees of their bipartite complements are 1. As the path is much more sparse, we want to extend this result to apply to graphs with larger maximum bipartite complement degree. We make the following definition in order to consider such graphs. 
\begin{definition}
Let $F(n, a, \Delta)$ be the minimum value of $F(G)$ over all $n$-vertex bipartite graphs $G$ that are subgraphs of $K_{a,n-a}$, such that the maximum degree of the bipartite complement of $G$ is $\Delta$.
\end{definition}
We can now generalize Theorem \ref{thm:matching_removed} as follows.
\begin{theorem}
\label{thm:3_part_induction}
For all nonnegative integers $n$, $a$, and $\Delta$ such that $n \geq 2 \Delta$ and $\Delta \leq a \leq n - \Delta$, $$F(n, a, \Delta) \geq {n - 2 \Delta \choose a - \Delta}.$$
\end{theorem}
\begin{proof}
We show the result by induction on $n$ and $a$, just as in Theorem \ref{thm:matching_removed}. For the base case, it suffices to note the trivial observation that $F(n, \Delta, \Delta) \geq 1$ and $F(n, n - \Delta, \Delta) \geq 1$ for all $n \geq 2 \Delta$.
For the inductive step, let $n > 2 \Delta$. Assume that for all $\Delta \leq d \leq n - 1 - \Delta$, $$F(n-1, d, \Delta) \geq {n - 1 - 2 \Delta \choose d - \Delta}.$$ It remains to be shown that for any $\Delta < a < n - \Delta$, $$F(n, a, \Delta) \geq {n - 2 \Delta \choose a - \Delta}.$$ Let $G$ be any bipartite graph with $n$ vertices that is a subgraph of $K_{a,n-a}$, such that the maximum degree of the bipartite complement of $G$ is $\Delta$. We show that $$F(G) \geq F(n-1, a-1, \Delta) + F(n-1, a, \Delta),$$ as then it would follow that $$F(G) \geq {n - 1 - 2 \Delta \choose a - 1 - \Delta} + {n - 1 - 2 \Delta \choose a - \Delta} = {n - 2 \Delta \choose a - \Delta}$$ by the inductive hypothesis.
First, note that $G$ has more than $2 \Delta$ vertices, so it must have at least one subset with more than $\Delta$ vertices, and therefore, because $\Delta$ is the maximum degree of the bipartite complement graph, $G$ must have at least one edge. Let this edge connect vertices $x$ and $y$ in the first and second subsets respectively. Removing a vertex from a graph cannot increase the maximum degree of the complement graph. Therefore, $F(G - \{x\}) \geq F(n-1, a-1, \Delta)$ and $F(G - \{y\}) \geq F(n-1, a, \Delta)$. It follows by definition that there exist pairwise $G$-different families $\mathcal{F}_x$ and $\mathcal{F}_y$ of permutations of $V(G) - \{x\}$ and $V(G) - \{y\}$ respectively such that $|\mathcal{F}_x| \geq F(n-1, a-1, \Delta)$ and $|\mathcal{F}_y| \geq F(n-1, a, \Delta)$. Let $\mathcal{F}$ be the family of permutations of $V(G)$ consisting of all permutations in $\mathcal{F}_x$ concatenated to $x$ and all permutations in $\mathcal{F}_y$ concatenated to $y$. Then $\mathcal{F}$ is pairwise $G$-different by construction, so $$F(G) \geq |\mathcal{F}| = F(n-1, a-1, \Delta) + F(n-1, a, \Delta).$$
\end{proof}
\begin{corollary}
\label{cor:constant_below_choose}
For all nonnegative integers $n$ and $\Delta$ such that $n \geq 2 \Delta$, $$F(n, \lfloor n/2 \rfloor, \Delta) \geq 2^{-2 \Delta} {n \choose \lfloor n/2 \rfloor}.$$
\end{corollary}
\begin{proof}
It is easy to see by expanding the binomial coefficient that 
\begin{align}
\begin{split}
F(n, \lfloor n/2 \rfloor, \Delta) &\geq {n - 2 \Delta \choose \lfloor n/2 \rfloor - \Delta} \\
&\geq 2^{-2 \Delta} {n \choose \lfloor n/2 \rfloor}. \\
\end{split}
\nonumber
\end{align}
\end{proof}

Although the lower bound in Theorem \ref{thm:3_part_induction} does not quite reach the upper bound given in Corollary \ref{cor:bipartite_upper_bound}, it comes within a constant factor of the upper bound for balanced bipartite graphs when $\Delta$ is a constant, as shown in Corollary \ref{cor:constant_below_choose}. This constant factor is due to the difficulty of finding sufficient base cases for the induction on $n$ and $a$ for large $\Delta$. Although for many $\Delta$ better base cases are easily found (as in $\Delta = 1$), it is difficult to find general base cases for all $\Delta$.

Because $\lim_{n \rightarrow \infty} \frac{1}{n} \log_2 {n \choose \lfloor n/2 \rfloor} = 1$ by Stirling's formula, the function $F(K_{\lfloor n/2 \rfloor, \lceil n/2 \rceil}) = {n \choose \lfloor n/2 \rfloor}$ grows exponentially on the order of $2^n$. We therefore remain primarily interested in showing that $F$ grows on the order of $2^n$ for various classes of balanced bipartite graphs, and thereby showing that the upper bound on $F$ of ${n \choose \lfloor n/2 \rfloor}$ is met asymptotically. We now use Corollary \ref{cor:constant_below_choose} to show that $F(n, \lfloor n/2 \rfloor, \Delta)$ grows on the order of $2^n$ if $\Delta = o(n)$.

\begin{theorem}
\label{thm:big_deg_comp}
$$\lim_{n \rightarrow \infty} \frac{1}{n} \log_2 F(n, \lfloor n/2 \rfloor, o(n)) = 1.$$
\end{theorem}
\begin{proof}
First note that 
\begin{align}
\begin{split}
\lim_{n \rightarrow \infty} \frac{1}{n} \log_2 F(n, \lfloor n/2 \rfloor, o(n))
&\leq \lim_{n \rightarrow \infty} \frac{1}{n} \log_2 F(K_{\lfloor n/2 \rfloor, \lceil n/2 \rceil}) = 1\\
\end{split}
\nonumber
\end{align}
by Stirling's formula. It therefore suffices to show the opposite inequality to prove a lower bound of 1. By Corollary \ref{cor:constant_below_choose}, 
\begin{align}
\begin{split}
\lim_{n \rightarrow \infty} \frac{1}{n} \log_2 F(n, \lfloor n/2 \rfloor, o(n))
&\geq \lim_{n \rightarrow \infty} \frac{1}{n} \log_2 \left( 2^{-2 \cdot o(n)} {n \choose \lfloor n/2 \rfloor} \right) \\
&\geq \lim_{n \rightarrow \infty} \frac{1}{n} \log_2 {n \choose \lfloor n/2 \rfloor} + \lim_{n \rightarrow \infty} \frac{1}{n} \log_2 2^{-2 \cdot o(n)} \\
&= 1 - \lim_{n \rightarrow \infty} \frac{2 \cdot o(n)}{n} \\
&= 1. \\
\end{split}
\nonumber
\end{align}
\end{proof}
This theorem is particularly interesting because it presents a very large class of graphs such that any graph $G$ in this class has the property that $F(G)$ is near $2^n$. However, as $\Delta = o(n)$, these graphs are relatively dense. In the next section we present specific but much sparser graphs $G$ for which $F(G)$ is near $2^n$.

\subsection{Union of Small Dense Balanced Bipartite Graphs}
\label{subsec:union_disjoint}
In this section we show that $F$ grows on the order of $2^n$ for graphs consisting of the union of small complete balanced bipartite graphs. 
We first present the following well-known lemma, which provides a method for placing lower bounds on $F(G)$ for graphs $G$ composed of disjoint subgraphs. An equivalent result is shown in \cite{korner_graph-different_2008}, but we present the proof as it is related to future proofs in this paper.

\begin{lemma}
\label{lem:concatenation}
Let $G$ be the union of disjoint graphs $G_1$ and $G_2$. Then $F(G) \geq F(G_1) \cdot F(G_2)$.
\end{lemma}
\begin{proof}
Let $\mathcal{F}_1 = \{\pi_1, \dots, \pi_{F(G_1)}\}$ and $\mathcal{F}_2 = \{\sigma_1, \dots, \sigma_{F(G_2)}\}$ be families of pairwise $G_1$-different and $G_2$-different permutations respectively of maximum size, so that $|\mathcal{F}_1| = F(G_1)$ and $|\mathcal{F}_2| = F(G_2)$. Then let $\mathcal{F}$ be the family of permutations consisting of $\pi_i$ concatenated to $\sigma_j$ for all $1 \leq i \leq F(G_1)$ and $1 \leq j \leq F(G_2)$. Then $\mathcal{F}$ is $G$-different, as for any two permutations $\pi_{i_1} \sigma_{j_1}$ and $\pi_{i_2} \sigma_{j_2}$, if $i_1 \neq i_2$, $\pi_{i_1}$ and $\pi_{i_2}$ are $G$-different; otherwise $j_1 \neq j_2$ and $\sigma_{j_1}$ and $\sigma_{j_2}$ are $G$-different. Therefore $F(G) \geq |\mathcal{F}| = F(G_1) \cdot F(G_2)$.
\end{proof}

Intuitively, if $F(G_1) \approx 2^{|V(G_1)|}$ and $F(G_2) \approx 2^{|V(G_2)|}$, then $F(G_1 + G_2) \approx 2^{|V(G_1)| + |V(G_2)|} = 2^{|V(G_1 + G_2)|}$ by Lemma \ref{lem:concatenation}. Because we want to find bipartite graphs $G$ for which $F(G) \approx 2^{|V(G)|}$, this idea is very useful, and forms the basis of the theorem below.
\begin{theorem}
\label{thm:disjoint_union}
Let $B(n,k(n))$ be the balanced bipartite graph of order $n$ consisting of the union of $k(n)$ disjoint balanced complete bipartite graphs, each of order $\lfloor n/k(n) \rfloor$ or $\lceil n/k(n) \rceil$. If $$k(n) = O \left( \frac{n}{\log_2 n} \right) ,$$ then $$\lim_{n \rightarrow \infty} \frac{1}{n} \log_2 F(B(n,k(n))) = 1.$$
\end{theorem}
\begin{proof}
For some given $n$, let $k = k(n)$, $B = B(n,k(n))$, and let the $k$ disjoint subgraphs of $B$ be $B_1, \dots, B_k$ with orders $n_1, \dots, n_k$ respectively. By Lemma \ref{lem:concatenation}, $$F(B) \geq \prod_{i=1}^k F(B_i) = \prod_{i = 1}^{k} {n_i \choose \lfloor n_i / 2 \rfloor}.$$ The right side can be expanded by Stirling's formula, which is easily applied to show that there exists a positive constant $l$ for which $${x \choose \lfloor x/2 \rfloor} \geq l \cdot \frac{2^x}{\sqrt{x}}$$ holds for all positive integers $x$. (The actual value of $l$ is not relevant to us, but it is easily bounded.) Then, as $\sum_{i=1}^{k} n_i = n$, and by the AM-GM inequality,
\begin{align}
\label{eq:union_lower_bound}
\begin{split}
F(B) &\geq \prod_{i = 1}^{k} {n_i \choose \lfloor n_i / 2 \rfloor}
\geq \prod_{i = 1}^{k} l \cdot \frac{2^{n_i}}{\sqrt{n_i}} 
= \frac{l^k \cdot 2^n}{\sqrt{\prod_{i = 1}^{k} n_i}} \\
&\geq \frac{l^k \cdot 2^n}{\sqrt{\left( \frac{n}{k} \right)^k}}
= \frac{l^k \cdot 2^n}{2^{\frac{k}{2} \log_2 \left( \frac{n}{k} \right)}} \\
&=  l^k \cdot 2^{n - \frac{k}{2} \log_2 \left( \frac{n}{k} \right)}. \\
\end{split}
\nonumber
\end{align}
Therefore 
\begin{equation}
\label{eq:disjoint_result}
F(B(n,k(n))) \geq l^{k(n)} \cdot 2^{n - \frac{k(n)}{2} \log_2 \left( \frac{n}{k(n)} \right)}.
\end{equation}
If $k(n) = O(n / \log_2 n)$, it is easily verifiable from (\ref{eq:disjoint_result}) that $$\lim_{n \rightarrow \infty} \frac{1}{n} \log_2 F(B(n,k(n))) \geq 1.$$ The opposite inequality is trivial as $B(n,k(n))$ is bipartite.
\end{proof}

Note that the proof of Theorem \ref{thm:disjoint_union} holds even if the $k(n)$ disjoint balanced complete bipartite graphs have different orders; however, the union is sparsest when the disjoint subgraphs are near in order.

Theorem \ref{thm:disjoint_union} provides an $n$-vertex bipartite graph $B(n,k(n))$ with maximum degree $O(\log_2 n)$ such that $F(B(n,k(n)))$ grows on the order of $2^n$, or more formally, such that $F(B(n,k(n))) = 2^{(1-o(1))n}$. This graph is the sparest balanced bipartite graph we currently know with this property.

\section{Families of Pairwise Matching-Different Permutations}
\label{sec:sparse_bipartite}

In Section \ref{sec:dense_bipartite}, we primarily dealt with relatively dense bipartite graphs $G$ such that $F(G)$ was near $2^n$. Now we examine very sparse bipartite graphs. The most prominently studied of these is the path; improvements on the lower bound on $F(P_n)$ were made in \cite{korner_pairwise_2006, korner_graph-different_2008, brightwell_permutation_2010}. In this paper we investigate $F$ for the matching graph on $n$ vertices, which we denote $M(n)$. (We will assume $n$ to be even whenever referencing $M(n)$ in this section.) As the matching is a subgraph of the path, any lower bounds on $F(M(n))$ would also apply to the $n$-vertex path. Additionally, the matching consists of the union of $n/2$ disjoint edges, giving it a special structure relating to Lemma \ref{lem:concatenation}.

We first generalize the function $F$ and show how the generalization is related to the original function.
\begin{definition}
\label{def:blank_spaces}
Let $F_b (G)$ be the maximum size of a family of pairwise $G$-different permutations of the vertices of $G$ with an additional $b$ blank spaces.
\end{definition}
Here a blank space can be thought of as an isolated vertex added to $G$. For example, consider the family $\mathcal{F}$ shown below.
\begin{equation}
\begin{matrix}
1 & 2 & * \\
* & 1 & 2 \\
2 & * & 1 \\
\end{matrix}
\nonumber
\end{equation}
We say $\mathcal{F}$ is family of 3  pairwise $M(2)$-different permutations of the vertices of $M(2)$ with 1 blank space; the blank space in each permutation is denoted by \lq$*$' and simply serves as a placeholder. By this definition, it is clear that $F_b (G) \leq F_c (G)$ if $b \leq c$ for any graph $G$. We now extend Definition \ref{def:blank_spaces} to account for families of permutations with unlimited blank spaces; an equivalent definition was made in \cite{korner_graph-different_2008}.
\begin{definition}
For any graph $G$, assign each element of $V(G)$ to a unique natural number. Let two infinite permutations of $\mathbb{N}$ be $G$-different if at some position their corresponding elements are both assigned to vertices in $G$ and form an edge in $G$. Then let $F_{\infty}(G)$ be the maximum size of a family of pairwise $G$-different infinite permutations of $\mathbb{N}$.
\end{definition}
K\"{o}rner, Malvenuto, and Simonyi \cite{korner_graph-different_2008} showed that for any graph $G$ on $n$ vertices, 
\begin{equation}
\label{eq:chromatic_upper_bound}
F_{\infty} (G) \leq (\chi(G))^{n},
\end{equation}
where $\chi(G)$ denotes the chromatic number of $G$. Therefore, for graphs with finitely many vertices, $F_{\infty} (G)$ is finite, so it follows that there exists a sufficiently large constant $b$ for which $F_b (G) = F_{\infty} (G)$. (For example, it is easy to verify that $b = n (\chi(G))^n$ satisfies this equation.) We can therefore think of $F_{\infty} (G)$ as the maximum size of a family of pairwise $G$-different permutations of the vertices of $G$ with arbitrarily many blank spaces, rather than in terms of infinite permutations of the natural numbers.

If $G(n)$ is a graph defined for all positive integer $n$, then let $$\rho_b(G) = \limsup_{n \rightarrow \infty} \frac{1}{n} \log_2 F_b(G(n)),$$ and let $\rho(G) = \rho_0(G)$. (In this paper $G(n)$ will usually be the first $n$ vertices of an infinite graph, as is the case with $M(n)$.) Therefore $\rho_b (G)$ measures the asymptotic behavior of $F_b(G(n))$. Although it may seem that $F_{\infty} (G(n))$ should be much larger than $F(G(n))$, the following two lemmas show that for certain graphs $G(n)$ such as the matching graph $M(n)$, $\rho(G)$ and $\rho_{\infty} (G)$ are equal. Very similar results were shown in \cite{korner_graph-different_2008, brightwell_permutation_2010} for the path; we present a generalization of these proofs below.

\begin{definition}
If $\mathcal{F}$ is a family of permutations of the vertices of some graph $G$ with any number of blank spaces (or with no blank spaces), let $H_{\mathcal{F}, G}$ be the graph whose vertices are the permutations in $\mathcal{F}$ and whose edges are all pairs of permutations $\sigma, \pi \in \mathcal{F}$ such that $\sigma$ and $\pi$ are $G$-different.
\end{definition}
Note that if $\mathcal{F}$ is pairwise $G$-different, then $H_{\mathcal{F}, G}$ is complete, or equivalently, $\alpha(H_{\mathcal{F}, G}) = 1$.

\begin{lemma}
\label{lem:recursion_no_blanks}
If $\mathcal{F}$ is a family of $G_0$-different permutations of the vertices of $G_0$ with unlimited blank spaces, then $\rho(G) \geq \frac{1}{|V(G_0)|} \log_2 |\mathcal{F}|$, where $G(n)$ consists of the union of $n/|V(G_0)|$ copies of $G_0$, for all $n$ which are divisible by $|V(G_0)|$.
\end{lemma}
We omit the proof of this lemma, which is similar to the proof of Lemma \ref{lem:concatenation}. A nearly identical result is shown in \cite{korner_graph-different_2008}, which is specific to the path but easily generalizable.

\begin{lemma}
\label{lem:blanks_dont_help}
Let $G(n)$ be a graph of order $n$ defined for all positive $n$ such that $G(n_1) + G(n_2)$ is a subgraph of $G(n_1 + n_2)$. If either $\rho(G)$ or $\rho_{\infty} (G)$ exists (that is, either of their limits exist and is not $\infty$), then both values exist and $\rho(G) = \rho_{\infty} (G).$
\end{lemma}
\begin{proof}
Clearly if $\rho_{\infty} (G)$ exists, then $\rho(G)$ exists and $\rho_{\infty} (G) \geq \rho(G)$, as $F_{\infty} (G(n)) \geq F(G(n))$ for all $n$. We now show by contradiction that if $\rho(G)$ exists, then $\rho_{\infty} (G)$ exists and $\rho(G) \geq \rho_{\infty} (G)$. Assume that $\rho(G)$ exists and is not $\infty$, but that $\rho_{\infty} (G) > \rho(G)$ or that $\rho_{\infty} (G) = \infty$. Then, in both of these cases, there exists an $N$ such that there is a family $\mathcal{F}$ of pairwise $G(N)$-different permutations of $V(G(N))$ with unlimited blanks, where $\frac{1}{N} \log_2 |\mathcal{F}| > \rho(G)$. However, by assumption the union of $k$ copies of $G(N)$ is a subgraph of $G(k N)$ for all positive integers $k$. Therefore, $\rho(G) \geq \frac{1}{N} \log_2 |\mathcal{F}|$ by Lemma \ref{lem:recursion_no_blanks}, which is a contradiction.
\end{proof}

Lemma \ref{lem:blanks_dont_help} shows that $\rho_{\infty} (M) = \rho(M).$ As $M(n)$ is bipartite, $\rho(M) \leq 1$, so $\rho_{\infty} (M) \leq 1$. Therefore $F_{\infty} (M(n))$ grows exponentially on the order of at most $2^n$. This bound was improved in \cite{korner_graph-different_2008}, where it was shown that $\sqrt{3}^n \leq F_{\infty} (M(n)) \leq 2^n$. The upper bound of $2^n$ was shown as part of the more general result that $F_{\infty} (G) \leq (\chi(G))^{|V(G)|}$, where $\chi(G)$ denotes the chromatic number of $G$. We use a different approach for bounding $F_{\infty} (M(n))$; we first bound $\alpha(H_{\mathcal{F}, M(2)})$ for families $\mathcal{F}$ of permutations of the vertices of $M(2)$, then we use this result to bound $\alpha(H_{\mathcal{F}, M(n)})$ for larger $n$. This approach helps determine $F_{\infty} (M(n))$ for small $n$ and provides a slightly stronger upper bound on $F_{\infty} (M(n))$ for all $n$ (better than $2^n$ by a constant factor). We also present some interesting constructions of families $\mathcal{F}$ with relatively good upper bounds on $\alpha(H_{\mathcal{F}, M(n)})$; these results mark progress towards potentially improving the lower bounds on $F_{\infty} (M(n))$.

\begin{lemma}
\label{lem:m2_ind_num}
Let $\mathcal{F}$ be a family of permutations of the vertices of $M(2)$ with $b$ blank spaces, and let $c = b + 2$ be the length of the permutations in $\mathcal{F}$. Then $$\alpha(H_{\mathcal{F}, M(2)}) \geq \frac{2^{c-2}}{2^c - 2} \cdot |\mathcal{F}|.$$
\end{lemma}
\begin{proof}
Assume the vertices of $M(2)$ are labeled 1 and 2, so that all permutations in $\mathcal{F}$ consist of 1, 2, and $c-2$ blanks. We first observe that an independent set in $H_{\mathcal{F}, M(2)}$ cannot contain permutations $\pi$ and $\sigma$ such that $\pi(j) = 1$ and $\sigma(j) = 2$ for some position $j$. Therefore, an independent set $I$ in $H_{\mathcal{F}, M(2)}$ is characterized by a string $s$ of 1's and 2's of length $c$. A permutation $\pi \in \mathcal{F}$ is in $I$ only if $\pi(j) = s(j)$ for every position $j$ at which $\pi$ does not have a blank space.

There are $2^c$ possible labelings of the $c$ positions with 1's and 2's, but 2 of these (all 1's and all 2's) always correspond to empty independent sets. Therefore let $I_1, \dots, I_{2^c-2}$ be the $2^c-2$ independent sets of maximal size in $H_{\mathcal{F}, M(2)}$ corresponding to strings of 1's and 2's of length $c$. Each permutation $\pi \in \mathcal{F}$ belongs to exactly $2^{c-2}$ of these independent sets, as each of the $c-2$ blank spaces in $\pi$ may be labeled 1 or 2 in the string $s$. Therefore $$\sum_{i=1}^{2^c-2} |I_i| = 2^{c-2} \cdot |\mathcal{F}|.$$ It follows by the pigeonhole principle that there exists some $k$ for which $$|I_k| \geq \frac{2^{c-2}}{2^c - 2} \cdot |\mathcal{F}|.$$
\end{proof}
\begin{corollary}
\label{cor:m2_quarter}
$$\alpha(H_{\mathcal{F}, M(2)}) > \frac14 \cdot |\mathcal{F}|.$$
\end{corollary}

The inequality in Lemma \ref{lem:m2_ind_num} is only significantly stronger than that in Corollary \ref{cor:m2_quarter} for families of permutations that are very short in length. It is therefore desirable to be able to consider families of permutations with as few blank spaces as possible. The following lemma shows that families of permutations with sufficiently many blank spaces can be condensed to equivalent families with fewer blank spaces.

\begin{lemma}
\label{lem:merge_cols}
Let $\mathcal{F} = \{\pi_1, \pi_2, \dots, \pi_p\}$ be a family of $p$ permutations of the vertices of $M(2)$ with $b$ blank spaces, and let $c = b + 2$ be the length of each permutation in $\mathcal{F}$. If $p < {c \choose 2}$, then there exists a family $\mathcal{F}' = \{\pi_1', \pi_2', \dots, \pi_p'\}$ of $p$ permutations of the vertices of $M(2)$ with $b-1$ blank spaces such that $H_{\mathcal{F}, M(2)}$ is a subgraph of $H_{\mathcal{F}', M(2)}$.
\end{lemma}
\begin{proof}
There are ${c \choose 2}$ pairs of positions $j_1, j_2$ in the permutations in $\mathcal{F}$. If $p = |\mathcal{F}| < {c \choose 2}$, then by the pigeonhole principle there must be some pair of positions $j_1, j_2$ ($1 \leq j_1 < j_2 \leq c$) such that each permutation in $\mathcal{F}$ has a blank space in at least one of these positions. Then for each $\pi_i$, let $\pi_i'$ be the permutation consisting of $\pi_i$ with the entry at position $j_2$ removed, and let $\pi_i'(j_1)$ take on the value of whichever of $\pi_i(j_1)$ or $\pi_i(j_2)$ is not a blank space. In other words, position $j_2$ was merged into position $j_1$ for each permutation $\pi_i$ to obtain $\pi_i'$. Then $\mathcal{F}' = \{\pi_1', \pi_2', \dots, \pi_p'\}$ satisfies the desired properties.
\end{proof}

We now apply Lemma \ref{lem:m2_ind_num} and Lemma \ref{lem:merge_cols} to determine the value of $F_{\infty} (M(4))$ and to improve the existing upper bound on $F_{\infty} (M(n))$ by a constant factor.

\begin{theorem}
\label{thm:m4_is_9}
$F_{\infty} (M(4)) = 9.$
\end{theorem}
\begin{proof}
We show that a family of 10 permutations of $V(M(2))$ with unlimited blank spaces cannot be $M(2)$-different. Specifically, it suffices to show that there is no family $\mathcal{F}_1$ of 10 permutations of the vertices of $M(2)$ with unlimited blanks such that $\alpha(H_{\mathcal{F}_1, M(2)}) \leq 3$ and $|E(H_{\mathcal{F}_1, M(2)})| > 22$. To see this, label the vertices on the two edges in $M(4)$ $(1, 2)$ and $(3, 4)$ respectively. Then, for any family $\mathcal{F} = \{\pi_1, \dots, \pi_{10}\}$ of 10 permutations of the vertices of $M(4)$ with unlimited blanks, let $\mathcal{F}_1 = \{\sigma_{1, 1}, \dots, \sigma_{1, 10}\}$ be the family $\mathcal{F}$ with all 3's and 4's replaced by blank spaces, and let $\mathcal{F}_2 = \{\sigma_{2, 1}, \dots, \sigma_{2, 10}\}$ be the family $\mathcal{F}$ with all 1's and 2's replaced by blank spaces. By this definition, $(\pi_{i_1}, \pi_{i_2}) \in E(H_{\mathcal{F}, M(4)})$ if and only if $(\sigma_{1, i_1}, \sigma_{1, i_2}) \in E(H_{\mathcal{F}_1, M(2)})$ or $(\sigma_{2, i_1}, \sigma_{2, i_2}) \in E(H_{\mathcal{F}_2, M(2)})$. Therefore, if $S \subseteq \{1, \dots, 10\}$ and if $\{\sigma_{1, i} : i \in S\}$ is an independent set in $H_{\mathcal{F}_1, M(2)}$, then $\{\sigma_{2, i} : i \in S\}$ must be a clique in $H_{\mathcal{F}_2, M(2)}$ in order for $H_{\mathcal{F}, M(2)}$ to be complete; the same applies for independent sets in $H_{\mathcal{F}_2, M(2)}$ and cliques in $H_{\mathcal{F}_1, M(2)}$. Because the largest clique in both $H_{\mathcal{F}_1, M(2)}$ and in $H_{\mathcal{F}_2, M(2)}$ has order at most $F_{\infty} (M(2)) = 3$, the independence number of both graphs must be 3 (it cannot be less than 3 by Lemma \ref{lem:m2_ind_num}). Furthermore, the complete graph on 10 vertices has 45 edges, so either $H_{\mathcal{F}_1, M(2)}$ or $H_{\mathcal{F}_2, M(2)}$ must have at least 23 edges in order for $H_{\mathcal{F}, M(2)}$ to be complete.

As ${5 \choose 2} = 10$, it is only necessary to consider families with at most 3 blank spaces by Lemma \ref{lem:merge_cols}. The case of 0 blanks is trivial; for families $\mathcal{F}_1$ of permutations of the vertices of $M(2)$ with 1 blank space, note that by Lemma \ref{lem:m2_ind_num}, $$\alpha(H_{\mathcal{F}_1, M(2)}) \geq \frac{2^{3 - 2}}{2^3 - 2} \cdot 10 = \frac{10}{3} > 3.$$ If $\mathcal{F}_1$ has 3 blank spaces, then the permutations have length 5, so each of the 10 pairs of positions $j_1, j_2$ must correspond to the 1 and the 2 of some permutation in $\mathcal{F}_1$; otherwise the family could be condensed by Lemma \ref{lem:merge_cols}. Therefore for each $1 \leq j \leq 5$, exactly 4 permutations in $\mathcal{F}_1$ do not have a blank space at position $j$. Among these 4 permutations, there are at most $2 \cdot 2 = 4$ pairs of $M(2)$-different permutations $(\pi_{i_1}, \pi_{i_2})$ which correspond to edges in $H_{\mathcal{F}_1, M(2)}$. Therefore $H_{\mathcal{F}_1, M(2)}$ has at most $5 \cdot 4 = 20$ edges. The only remaining case is when the permutations have 2 blank spaces. We used a brute force computer search for this case, and found that no 10 permutations of the vertices of $M(2)$ with 2 blanks has independence number 3.
\end{proof}
\begin{corollary}
\label{cor:m4_value}
$F_{\infty} (M(n)) < 9 \cdot 2^{n-4}$ for even $n > 4$.
\end{corollary}
\begin{proof}
It suffices to show that for even $n > 4$, $F_{\infty} (M(n)) < 4 \cdot F_{\infty} (M(n-2))$. To show this, we use the same method of separating out the independent edges that we used in Theorem \ref{thm:m4_is_9}. Label the vertices of the $n/2$ edges in $M(n)$ $(1,2), (3,4), \dots, (n-1,n)$. Let $\mathcal{F}$ be some family of pairwise $M(n)$-different permutations of the vertices of $M(n)$ with unlimited blank spaces. Let $\mathcal{F}_1$ be the family $\mathcal{F}$ with all numbers other than 1's and 2's replaced with blank spaces, and let $\mathcal{F}_2$ be the family $\mathcal{F}$ with all 1's and 2's replaced with blank spaces. (By this definition, $\mathcal{F}_1$ contains permutations of the vertices of $M(2)$ and $\mathcal{F}_2$ contains permutations of the vertices of $M(n-2)$). Because any independent set in $H_{\mathcal{F}_1, M(2)}$ must correspond to a clique of equal size in $H_{\mathcal{F}_2, M(n-2)}$, the clique number of $H_{\mathcal{F}_2, M(n-2)}$ must be at least $\alpha(H_{\mathcal{F}_1, M(2)}) > \frac14 \cdot |\mathcal{F}|$. By definition, $F_{\infty} (M(n-2))$ is an upper bound on the clique number of $H_{\mathcal{F}_2, M(n-2)}$, so $$F_{\infty} (M(n-2)) > \frac14 \cdot |\mathcal{F}|.$$ As this inequality holds for all pairwise $M(n)$-different families $\mathcal{F}$ of permutations of the vertices of $M(n)$ with unlimited blanks, it must be that $$F_{\infty} (M(n-2)) > \frac14 \cdot F_{\infty} (M(n)).$$ 
\end{proof}

To conclude this section we present results which were motivated by the problem of improving the lower bound on $F_{\infty} (M(n))$. We first observe that there exist families $\mathcal{F}$ of permutations of the vertices of $M(2)$ such that $\alpha(H_{\mathcal{F}, M(2)})$ is very close to $\frac14 \cdot |\mathcal{F}|$. Specifically, for any integer $c \geq 2$, let $\mathcal{A}_c$ be the family of the $c(c-1)$ distinct permutations of the vertices of $M(2)$ with $c-2$ blank spaces. Let $s$ be a string of 1's and 2's of length $c$ characterizing an independent set $I$ in $H_{\mathcal{A}_c, M(2)}$. If $s$ has $x$ 1's and $y$ 2's, then $|I| = xy$ by the definition of $\mathcal{F}$. Therefore the size of the largest independent set in $H_{\mathcal{A}_c, M(2)}$ is $\alpha(H_{\mathcal{A}_c, M(2)}) = \lfloor c/2 \rfloor \cdot \lceil c/2 \rceil$. It follows that 
\begin{equation}
\label{eq:upper_bound_reached}
\lim_{c \rightarrow \infty} \frac{\alpha(H_{\mathcal{A}_c, M(2)})}{|\mathcal{A}_c|} = \lim_{c \rightarrow \infty} \frac{\lfloor c/2 \rfloor \cdot \lceil c/2 \rceil}{c(c-1)} = \frac14.
\end{equation}
This construction shows that the bound in Lemma \ref{lem:m2_ind_num} is nearly optimal. We now apply the ideas in Lemma \ref{lem:m2_ind_num} and in (\ref{eq:upper_bound_reached}) to get an interesting result.

Let $\mathcal{F}$ be a family of $p$ pairwise $M(n)$-different permutations of the vertices of $M(n)$ with unlimited blank spaces, and once again label the vertices of the edges in $M(n)$ $(1,2), (3,4), \dots, (n-1, n)$. Let $\mathcal{E}_2, \mathcal{E}_4, \dots, \mathcal{E}_n$ be defined such that $\mathcal{E}_k = \{\sigma_{k, 1}, \dots, \sigma_{k, p}\}$ consists of the family $\mathcal{F}$ with all non-blank entries other than $k-1$ and $k$ replaced by blank spaces in each permutation. Let $\mathcal{F}_0 = \{\pi_{0, 1}, \dots, \pi_{0, p}\}$ be a family of $p$ empty permutations (or permutations of the vertices of $M(0)$). It follows that $H_{\mathcal{F}_0, M(0)}$ is empty and $\alpha(H_{\mathcal{F}_0, M(0)}) = p$. Then let $\mathcal{F}_2 = \{\pi_{2, 1}, \dots, \pi_{2, p}\}$ be defined so that $\pi_{2, j}$ consists of $\pi_{0, j}$ concatenated to $\sigma_{2, j}$, and in general, let $\mathcal{F}_k = \{\pi_{k, 1}, \dots, \pi_{k, p}\}$ be such that $\pi_{k, j}$ consists of $\pi_{k-2, j}$ concatenated to $\sigma_{k, j}$. (This definition is such that $H_{\mathcal{F}_n, M(n)} = H_{\mathcal{F}, M(n)}$.) Note that for any positive even $k$ and for any indices $i_1$ and $i_2$, $(\pi_{k, i_1}, \pi_{k, i_2}) \in E(H_{\mathcal{F}_k, M(k)})$ if and only if either $(\pi_{k-2, i_1}, \pi_{k-2, i_2}) \in E(H_{\mathcal{F}_{k-2}, M(k-2)})$ or $(\sigma_{k, i_1}, \sigma_{k, i_2}) \in E(H_{\mathcal{E}_k, M(2)})$. It follows by Lemma \ref{lem:m2_ind_num} that for any independent set $I$ in $H_{\mathcal{F}_{k-2}, M(k-2)}$, there exists an independent set $I'$ in $H_{\mathcal{F}_k, M(k)}$, which is a subset of $I$, such that $|I'| > \frac14 \cdot |I|$. Therefore $\alpha(H_{\mathcal{F}_k, M(k)}) > \frac14 \cdot \alpha(H_{\mathcal{F}_{k-2}, M(k-2)})$, so by induction $$\alpha(H_{\mathcal{F}_k, M(k)}) > \frac{p}{2^k}$$ for all positive even $k$. As shown in (\ref{eq:upper_bound_reached}), for $k = 2$ there exist families $\mathcal{F}_2$ (namely, $\mathcal{A}_c$ for large $c$) such that $\alpha(H_{\mathcal{F}_2, M(2)})$ is arbitrarily close to $\frac{1}{2^k} = \frac14$. However, if $p \approx 2^n$, then $\alpha(H_{\mathcal{F}_k, M(k)})$ must be approximately $\frac{p}{2^k}$ for all $1 \leq k \leq n$, as $\alpha(H_{\mathcal{F}_n, M(n)}) = 1$. More generally, if $\rho_{\infty}(M) \geq \frac12 \log_2 a$ for some constant $a$, then there must exist families $\mathcal{E}_2, \dots, \mathcal{E}_n$ such that their corresponding families $\mathcal{F}_0, \dots, \mathcal{F}_n$ satisfy $\alpha(H_{\mathcal{F}_k, M(k)}) \approx \frac{p}{a^{k/2}}$ and $n \approx \log_{\sqrt{a}} p$ (loosely speaking). Below, we present a construction which partially answers this question by providing families $\mathcal{E}_2, \dots, \mathcal{E}_l$ such that for certain $a > 3$, $\alpha(H_{\mathcal{F}_k, M(k)})$ is within a constant factor of $\frac{p}{a^{k/2}}$ for $0 \leq k \leq l$, where $l$ grows logarithmically as a function of $p$ (but slower than $\log_{\sqrt{a}} p$). We later explain how this result could potentially be extended to improve the lower bound on $\rho_{\infty} (M)$.

\begin{theorem}
\label{thm:reducing_ind_num}
Let $\mathcal{A}$ be some family of permutations of the vertices of $M(2)$ with unlimited blank spaces, and let $p$ be some positive integer. Let $l = 2 \cdot \lceil \log_{|\mathcal{A}|} p \rceil$. Then there exists a family $\mathcal{F}_l$ of $p$ permutations of $M(l)$ such that $$\alpha(H_{\mathcal{F}_l, M(l)}) \leq P \cdot a^{l/2},$$ where $P$ is the least power of $|\mathcal{A}|$ not less than $p$ and $a = \dfrac{\alpha(H_{\mathcal{A}, M(2)})}{|\mathcal{A}|}$.
\end{theorem}
\begin{proof}
Let $q = |\mathcal{A}|$ and let $\mathcal{A} = \{A_1, A_2, \dots, A_q\}$. For even $k$ where $2 \leq k \leq l$, let $\mathcal{E}_k = \{\sigma_{k, 1}, \dots, \sigma_{k, P}\}$ consist of the pattern
\begin{equation}
\underbrace{A_1, \dots, A_1}_{q^{k/2-1}}, \underbrace{A_2, \dots, A_2}_{q^{k/2-1}} \; \cdots \; \underbrace{A_q, \dots, A_q}_{q^{k/2-1}}
\nonumber
\end{equation}
repeated $\frac{P}{q^{k/2}}$ times. Let $\mathcal{F}_k = \{\pi_{k, 1}, \dots, \pi_{k, P}\}$ be defined the same before: $\pi_{k, j} = \pi_{k-2, j} \; \sigma_{k, j}$.

It remains to be shown that $\alpha(H_{\mathcal{F}_l, M(l)}) = P \cdot a^{l/2}$. We first observe that any independent set with indices $I_l$ of maximum size in $H_{\mathcal{F}_l, M(l)}$ is constructed in the following manner. Let $I_0 = \{1, \dots, P\}$ represent the indices of the independent set $\pi_{0, 1}, \dots, \pi_{0, P}$ in $H_{\mathcal{F}_0, M(0)}$. For each $2 \leq k \leq l$, choose some independent set $\mathcal{B}_k$ in $H_{\mathcal{A}, M(2)}$ of size $\alpha(H_{\mathcal{A}, M(2)})$. Then let $I_k = \{j \in I_{k-2} : \sigma_{k, j} \in \mathcal{B}_k\}$. By this construction, $\alpha(H_{\mathcal{A}, M(2)})$ out of every $q$ elements of $I_{k-2}$ will be in $I_k$. Therefore by induction, as $I_k$ is of maximum size by assumption, $$\alpha(H_{\mathcal{F}_k, M(k)}) = |I_k| = P \cdot (\alpha(H_{\mathcal{A}, M(2)}) / q)^{k/2} = P \cdot a^{k/2}$$ for all even $0 \leq k \leq l$.
\end{proof}

\section{Conclusion}
\label{sec:conclusion}
In this paper, we develop new methods for bounding the maximum size of a family of pairwise graph-different permutations for various bipartite graphs. For specific non-complete bipartite graphs $G$ with vertex subsets of size $a$ and $b$, we show that the upper bound on $F(G)$ of $F(K_{a,b})$ is tight. We show that if $G(n)$ is any balanced bipartite graph on $n$ vertices with minimum degree $n/2 - o(n)$, then $F(G(n))$ grows on the same exponential order as $F(K_{\lfloor n/2 \rfloor, \lceil n/2 \rceil})$ when $n \rightarrow \infty$. We also show that this growth is achieved for certain much sparser balanced bipartite graphs. We present several new bounds on $F_{\infty}$ for the matching graph $M(n)$. Specifically, we determine the exact value of $F_{\infty} (M(4))$, and improve the upper bound on $F_{\infty} (M(n))$. Our new methods and bounds make potential progress towards determining the value of $F(P_n)$.

%

\section{Acknowledgements}
We would like to thank Dr. Tanya Khovanova for her helpful comments and suggestions on the paper. We would also like to thank the MIT PRIMES program for providing us with the opportunity to perform this research.

\bibliographystyle{ieeetr}
\singlespacing
\bibliography{library.bib}
\end{document}